\documentclass{amsart}[12pt]
\usepackage{amsmath, amsthm, amscd, amsfonts,}
\usepackage{graphicx,float}

\usepackage{amssymb}

\usepackage{pb-diagram}

\def\l{\lambda}

\newtheorem{theorem}{Theorem}[section]
\newtheorem{lemma}[theorem]{Lemma}

\newtheorem{problem}[theorem]{Problem}
\newtheorem{definition}[theorem]{Definition}

\newtheorem{remark}[theorem]{Remark}

\newtheorem{claim}[theorem]{Claim}
\numberwithin{equation}{section}

\def\l{\lambda}

\def\l{\lambda}

\def\rmark{\mbox{$\rm\bf\rule{0.06em}{1.45ex}\kern-0.05em R$}}
\def\pmark{\mbox{$\rm\bf\rule{0.06em}{1.45ex}\kern-0.05em P$}}
\def\nmark{\mbox{$\rm\bf\rule{0.06em}{1.45ex}\kern-0.05em N$}}
\def\vdash{\mbox{$\rm\| \kern-0.13em -$}}
\newcommand{\lusim}[1]{\smash{\underset{\raisebox{1.2pt}[0cm][0cm]{$\sim$}}
{{#1}}}}

\usepackage{pb-diagram}

\def\l{\lambda}

\def\rmark{\mbox{$\rm\bf\rule{0.06em}{1.45ex}\kern-0.05em R$}}
\def\pmark{\mbox{$\rm\bf\rule{0.06em}{1.45ex}\kern-0.05em P$}}
\def\nmark{\mbox{$\rm\bf\rule{0.06em}{1.45ex}\kern-0.05em N$}}
\def\vdash{\mbox{$\rm\| \kern-0.13em -$}}

\begin{document}

\title[Shelah's strong covering property]{Shelah's strong covering property and $CH$ in $V[r]$ }

\author[E. Eslami and M. Golshani]{Esfandiar Eslami and Mohammad Golshani}

\thanks{The authors would like to thank Prof. Sy Friedman for some useful
comments and remarks concerning $\S4.$ In particular Theorem 4.2
is suggested by him.} \maketitle




\begin{abstract}
In this paper we review Shelah's strong covering property and its
applications. We also extend some of the results of Shelah and
Woodin on the failure of $CH$ by adding a real.
\end{abstract}

\maketitle

\section{Introduction}

In this paper we review Shelah's strong covering property and its
applications, in particular, to pairs $(W,V)$ of models of set
theory with $V=W[r],$ for some real $r$. We also consider the
consistency results of Shelah and Woodin  on the failure of $CH$
by adding a real and prove some related results. Some other
results are obtained too.

The structure of the paper is as follows: In $\S 2$ we present an
interesting result of Vanliere [6] on blowing up the continuum
with a real. In $ \S 3$ we give some applications of Shelah's
strong covering property. In $\S 4$ we consider the work of Shelah
and Woodin stated above and prove some new results. Finally in $\S
5$ we state some problems.

\section{On a theorem of Vanliere}

In this section we prove the following result of Vanliere [6]:

\begin{theorem} Assume $V=\mathbf{L}[X,a]$ where $X \subseteq \omega_n$
for some $n < \omega,$ and $a \subseteq \omega.$ If $\mathbf{L}[X] \models
 ZFC+GCH $ and the cardinals of $\mathbf{L}[X]$ are the
true cardinals, then $GCH$ holds in $V$.
\end{theorem}
\begin{proof} Let $\kappa$ be an infinite cardinal. We prove the
following:

 $(*)_{\kappa}:$ $\hspace{2.cm}$For any $Y \subseteq \kappa$ there
is an ordinal $\alpha < \kappa^+$ and

$\hspace{2.85cm}$ a set $Z \in \mathbf{L}[X], Z \subseteq \kappa$ such that $Y \in
\mathbf{L}_{\alpha}[Z,a].$

Then it will follow that $\mathcal{P}(\kappa) \subseteq \bigcup_{\alpha <
\kappa^+} \bigcup_{Z \in \mathcal{P}^{\mathbf{L}[X]}(\kappa)} \mathbf{L}_{\alpha}[Z,a],$ and
hence

\begin{center}

$2^{\kappa} \leq \sum_{\alpha < \kappa^+} \sum_{Z \in
\mathcal{P}^{\mathbf{L}[X]}(\kappa)}| \mathbf{L}_{\alpha}[Z,a]| \leq \kappa^+.
(2^{\kappa})^{\mathbf{L}[X]}. \kappa= \kappa^+$

\end{center}

which gives the result. Now we return to the proof of
$(*)_{\kappa}$.

{\bf Case 1}. $\kappa \geq \aleph_n$.

Let $Y \subseteq \kappa.$ Let $\theta$ be large enough regular
such that $Y \in \mathbf{L}_{\theta}[X,a].$ Let $N \prec \mathbf{L}_{\theta}[X,a]$
be such that $| N |= \kappa, N \cap \kappa^{+} \in
\kappa^{+}$ and $\kappa \cup \{Y, X, a\} \subseteq N$. By the
condensation lemma there are $\alpha < \kappa^+$ and $\pi$ such
that $\pi: N \cong \mathbf{L}_{\alpha}[X, a].$ then $Y=\pi(Y) \in
\mathbf{L}_{\alpha}[X, a].$ Thus $(*)_{\kappa}$ follows.

{\bf Case 2.} $\kappa < \aleph_n$.

We note that the above argument does not work in this case. Thus
another approach is needed. To continue the work, we state a
general result (again due to Vanliere) which is of interest in its
own sake.

\begin{lemma} Suppose $\mu \leq \kappa < \lambda \leq \nu$ are
infinite cardinals, $\lambda$ regular. Suppose that $a \subseteq
\mu, Y \subseteq \kappa, Z \subseteq \lambda,$ and $X \subseteq
\nu$ are such that $V=\mathbf{L}[X,a], Z \in \mathbf{L}[X], Y \in \mathbf{L}[Z,a]$ and
$\lambda_{\mathbf{L}[X]}^+= \lambda^+.$ Then there exists a proper initial
segment $Z^{'}$ of $Z$ such that $Z^{'} \in \mathbf{L}[X]$ and $Y \in
\mathbf{L}[Z^{'},a].$
\end{lemma}
\begin{proof} Let $\theta \geq \nu$ be regular such that $Y \in
\mathbf{L}_{\theta}[Z,a].$ Let $N \prec \mathbf{L}_{\theta}[Z,a]$ be such that $|
N |= \lambda, N \cap \lambda^{+} \in \lambda^{+}$ and $\lambda
\cup \{Y, Z, a\} \subseteq N$. By the condensation lemma we can
find $\delta < \lambda^+$ and $\pi$ such that $\pi: N \cong
\mathbf{L}_{\delta}[Z, a].$

In $V$, let $ \langle M_{i}: i < \lambda \rangle$ be a continuous
chain of elementary submodels of  $\mathbf{L}_{\delta}[Z, a]$  with union
$\mathbf{L}_{\delta}[Z, a]$ such that for each $i < \lambda, M_{i}
\supseteq \kappa$, $| M_{i}| < \lambda$ and $M_{i} \cap
\lambda \in \lambda.$

In $\mathbf{L}[Z]$ let  $ \langle W_{i}: i < \lambda \rangle$ be a
continuous chain of elementary submodels of  $\mathbf{L}_{\delta}[Z]$ with
union $\mathbf{L}_{\delta}[Z]$ such that for each $i < \lambda, W_{i}
\supseteq \kappa$,  $| W_{i}| < \lambda$ and $W_{i} \cap
\lambda \in \lambda$

Now we work in $V$. Let $E= \{i < \lambda: M_{i}\cap
\mathbf{L}_{\delta}[Z]=W_{i} \}$. Then $E$ is a club of $\lambda.$ Pick $i
\in E$ such that $Y \in M_{i},$ and let $M=M_{i},$ and  $W=W_{i}.$
By the condensation lemma let $\eta < \lambda$ and $\bar{\pi}$ be
such that $\bar{\pi}:M \cong \mathbf{L}_{\eta}[Z^{'}, a]$ where $Z^{'}=
\bar{\pi}[M \cap Z]=\bar{\pi}[(M \cap \lambda) \cap Z]=(M \cap
\lambda) \cap Z$, a proper initial segment of Z. Then
$Y=\bar{\pi}(Y) \in \mathbf{L}_{\eta}[Z^{'}, a]$ and $Z^{'} \subseteq \eta
< \lambda.$ It remains to observe that $Z^{'} \in \mathbf{L}[X]$ as $Z^{'}$
is an initial segment of Z. The lemma follows.
\end{proof}

We are now ready to complete the proof of Case 2. By Lemma 2.2 we
can find a bounded subset $X_n$ of $\omega_n$ such that $X_{n} \in
\mathbf{L}[X]$ and $Y \in \mathbf{L}[X_{n},a]$. Now trivially we can find a subset
$Z_{n-1}$ of $\omega_{n-1}$ such that $\mathbf{L}[X_{n}]=\mathbf{L}[Z_{n-1}],$ and
hence $Z_{n-1} \in \mathbf{L}[X]$ and $Y \in \mathbf{L}[Z_{n-1},a].$ Again by Lemma
2.2 we can find a bounded subset $X_{n-1}$ of $\omega_{n-1}$ such
that $X_{n-1} \in \mathbf{L}[X]$ and $Y \in \mathbf{L}[X_{n-1},a],$ and then we find
a subset $Z_{n-2}$ of $\omega_{n-2}$ such that
$\mathbf{L}[X_{n-1}]=\mathbf{L}[Z_{n-2}]$. In this way we can finally find a subset
$Z$ of $\kappa$ such that $Z \in \mathbf{L}[X]$ and $Y \in \mathbf{L}[Z,a]$. Then as
in case 1, for some $\alpha < \kappa^+, Y \in \mathbf{L}_{\alpha}[Z,a]$ and
$(*)_{\kappa}$ follows.
\end{proof}

\section{Shelah's strong covering property and its applications}

In this section we study Shelah's strong covering property and
give some of its applications.
 By a pair $(W,V)$ we always mean a pair $(W,V)$ of
models of $ZFC$ with the same ordinals such that $W \subseteq V.$
Let us give the main definition.
\begin{definition}
$(1)$ $(W,V)$ satisfies the strong $(\lambda,
\alpha)-$covering property, where $\lambda$ is a regular cardinal
of $V$ and $\alpha$ is an ordinal, if for every model $M \in V$
with universe $\alpha$ (in a countable language) and $a \subseteq
\alpha, | a | < \lambda$ (in $V$), there is $b \in W$ such
that $a \subseteq b \subseteq \alpha, b \prec M,$ and $ | b
| < \lambda$ (in $V$). $(W,V)$ satisfies the strong
$\lambda-$covering property if it satisfies the strong $(\lambda,
\alpha)-$covering property for every $\alpha.$

$(2)$ $(W,V)$ satisfies the strong $(\lambda^*, \lambda, \kappa,
\mu)-$covering property, where $\lambda^{*} \geq \lambda \geq
\kappa$ are regular cardinals of $V$ and $\mu$ is an ordinal, if
player one has a winning strategy in the following game, called
the $(\lambda^*, \lambda, \kappa, \mu)-$covering game, of length
$\lambda$:

In the $i-$th move player I chooses $a_{i} \in V$ such that $a_{i}
\subseteq \mu, | a_{i} | < \lambda^{*}$ (in $V$) and
$\bigcup_{j < i}b_{j} \subseteq a_{i}$, and player II chooses
$b_{i} \in V$ such that $b_{i} \subseteq \mu, | b_{i} | <
\lambda^{*}$ (in $V$) and $\bigcup_{j \leq i}a_{j} \subseteq
b_{i}$.

Player I wins if there is a club $C \subseteq \lambda$ such that
for every $\delta \in C \cup \{ \lambda \}, cf(\delta)= \kappa
\Rightarrow \bigcup_{i < \delta}a_{i} \in W.$ $(W,V)$ satisfies
the strong $(\lambda^*, \lambda, \kappa, \infty)-$covering
property, if it satisfies the strong $(\lambda^*, \lambda, \kappa,
\mu)-$covering property for every $\mu.$
\end{definition}

The following theorem shows the importance of the first part of
this definition and plays an important role in this section.

\begin{theorem} Suppose $V=W[r], r$ a real and $(W,V)$
satisfies the strong $(\lambda, \alpha)-$covering property for
$\alpha < ([(2^{< \lambda})^{W}]^{+})^{V}$. Then $(2^{<
\lambda})^{V}=| (2^{< \lambda})^{W}| ^{V}.$
\end{theorem}
\begin{proof} Cf. [3, Theorem VII.4.5.].
\end{proof}

It follows from Theorem 3.2 that if $V=W[r], r$ a real and $(W,V)$
satisfies the strong $(\lambda^{+},
([(2^{\lambda})^{W}]^{+})^{V})-$covering property, then $(2^{
\lambda})^{V}=|(2^{\lambda})^{W}| ^{V}.$

We are now ready to give the applications of the strong covering
property. For a pair $(W,V)$ of models of $ZFC$ consider the
following conditions:

$\hspace{1cm} (1)_{\kappa}:\bullet$ $V=W[r], r$ a real,

$\hspace{2cm}\bullet$ $V$ and $W$ have the same cardinals $\leq
\kappa^{+},$

$\hspace{2cm}\bullet$ $W \models  \forall \lambda \leq
\kappa, 2^{\lambda}= \lambda^{+} $

$\hspace{2cm}\bullet$ $V \models  2^{\kappa} > \kappa^{+}
$.

$\hspace{1cm} (2)_{\kappa}:$ $W \models  GCH $.

$\hspace{1cm} (3)_{\kappa}:$ $V$ and $W$ have the same cardinals.

\begin{theorem} $(1)$ Suppose there is a pair $(W,V)$
satisfying $(1)_{\aleph_0}$ and $(2)_{\aleph_0}$. Then
$\aleph_2^{V}$ in inaccessible in $\mathbf{L}$.

$(2)$ Suppose there is a pair $(W,V)$ as in $(1)$ with $V \models
 2^{\aleph_0}> \aleph_2$. Then $0^{\sharp} \in
V$.

$(3)$ Suppose there is a pair $(W,V)$ as in $(1)$ with $CARD^{W}
\cap (\aleph_1^{V}, \aleph_2^{V})= \emptyset$. Then $0^{\sharp}
\in V$.

$(4)$ Suppose $\kappa > \aleph_0$ and there is a pair $(W,V)$
satisfying $(1)_{\kappa}.$ Then $0^{\sharp} \in V.$
\end{theorem}
Before we give the proof of Theorem 3.3 we state some conditions
which imply Shelah's strong covering property. Suppose that in
$V,0^{\sharp}$ does not exist. Then:

$\hspace{1cm} (\alpha)$ If $\lambda^{*} \geq \aleph_2^{V}$ is
regular in $V$, then $(W,V)$ satisfies the strong
$\lambda^{*}-$covering

$\hspace{1.6cm}$property.

$\hspace{1cm} (\beta)$ If $CARD^{W} \cap (\aleph_1^{V},
\aleph_2^{V})= \emptyset$ then $(W,V)$ satisfies the strong
$\aleph_1^{V}-$covering

$\hspace{1.6cm}$property.

\begin{remark} For $\lambda^{*} \geq \aleph_3^{V}, (\alpha)$
follows from [3, Theorem VII.2.6], and $(\beta)$ follows from [3,
Theorem VII.2.8]. In order to obtain $(\alpha)$ for $\lambda^{*}=
\aleph_2^{V}$ we can proceed as follows: As in the proof of
[3, Theorem VII.2.6] proceed by induction on $\mu$ to show that
$(\mathbf{L},V)$ satisfies the strong $(\aleph_2^{V},\aleph_1^{V},
\aleph_0^{V}, \mu)-$covering property. For successor $\mu$ (in
$\mathbf{L}$) use [3, Lemma VII.2.2] and for limit $\mu$ use [3, Remark VII.2.4] (instead of [3, Lemma VII.2.3]). It then follows that $(\mathbf{L},V)$
and hence $(W,V)$ satisfies the strong $\aleph_2^{V}-$covering
property.
\end{remark}

\begin{proof} $(1)$ We may suppose that $0^{\sharp}
\notin V.$ Then by $(\alpha), (W,V)$ satisfies the strong
$\aleph_2^{V}-$covering property. On the other hand by Jensen's
covering lemma and [3, Claim VII.1.11], $W$ has squares. By [3, Theorem
VII.4.10], $\aleph_2^{V}$ is inaccessible in $W$, and hence in
$\mathbf{L}$.

$(2)$ Suppose not. Then by $(\alpha), (W,V)$ satisfies the strong
$\aleph_2^{V}-$covering property. By Theorem 3.2,
$(2^{\aleph_0})^{V} \leq (2^{\aleph_1})^{V}=
|(2^{\aleph_1})^{W}| ^{V}=|
\aleph_2^{W}|=\aleph_2^{V},$ which is a contradiction.

$(3)$ Suppose not. Then by $(\beta), (W,V)$ satisfies the strong
$\aleph_1^{V}-$covering property, hence by Theorem 3.2,
$(2^{\aleph_0})^{V}= |(2^{\aleph_0})^{W}|
^{V}=\aleph_1^{V},$ which is a contradiction.

$(4)$ Suppose not. Then by $(\alpha), (W,V)$ satisfies the strong
$\kappa^{+}-$covering property. By Theorem 3.2, $(2^{\kappa})^{V}=
|(2^{\kappa})^{W}| ^{V}=\kappa^{+},$ and we get a
contradiction.
\end{proof}
\begin{theorem} $(1)$ Suppose there is a pair $(W,V)$
satisfying $(1)_{\kappa},(2)_{\kappa}$ and $(3)_{\kappa}.$ Then
there is in $V$ an inner model with a measurable cardinal.

$(2)$ Suppose there is a pair $(W,V)$ satisfying $(1)_{\kappa},$
where $\kappa \geq \aleph_{\omega}.$ Further suppose that
$\kappa_{W}^{++}=\kappa_{V}^{++}$ and $(W,V)$ satisfies the
$\kappa^{+}-$covering property. Then there is in $V$ an inner
model with a measurable cardinal.
\end{theorem}
\begin{proof} $(1)$ Suppose not. Then by [3, conclusion VII.4.3(2)],
$(W,V)$ satisfies the strong $\kappa^{+}-$covering property, hence
by Theorem 3.2, $(2^{\kappa})^{V}=|(2^{\kappa})^{W} | ^{V}=
\kappa^{+},$ which is a contradiction.

$(2)$ Suppose not. Let $\kappa=\mu^{+n},$ where $\mu$ is a limit
cardinal, and $n < \omega.$ By [3, Theorem VII.2.6, Theorem VII.4.2(2) and
Conclusion VII.4.3(3)], we can show that $(W,V)$ satisfies the
strong $(\kappa^{+}, \kappa, \aleph_{1}, \mu)-$covering property.
On the other hand since $(W,V)$ satisfies the
$\kappa^{+}-$covering property and $V$ and $W$ have the same
cardinals $\leq \kappa^{+}, (W,V)$ satisfies the
$\mu^{+i}-$covering property for each $i \leq n+1.$ By repeatedly
use of [3, Lemma VII.2.2], $(W,V)$ satisfies the strong
$(\kappa^{+}, \kappa, \aleph_{1}, \kappa^{++})-$covering property,
and hence the strong  $(\kappa^{+}, \kappa^{++})-$covering
property. By Theorem 3.2, $(2^{\kappa})^{V}=|(2^{\kappa})^{W}
| ^{V}= \kappa^{+},$ which is a contradiction.
\end{proof}
\begin{remark}
In [3] (see also [4]), Theorem 3.4(1), for $\kappa=
\aleph_0,$ is stated under the additional assumption $2^{\aleph_0}
> \aleph_{\omega}$ in $V$.
\end{remark}
Let us close this section by noting that the hypotheses in
[3, Conclusion VII.4.6] are inconsistent. In other words we are
going to show that the following hypotheses are not consistent:

$\hspace{1cm}(a)$ $V$ has no inner model with a measurable
cardinal.

$\hspace{1cm}(b)$ $V=W[r], r$ a real, $V$ and $W$ have the same
cardinals $\leq \lambda,$ where $(2^{\aleph_0})^{V} \geq$

$\hspace{1.5cm}$ $\lambda \geq\aleph_{\omega}^{V}, \lambda$ is a
limit cardinal.

$\hspace{1cm}(c)$ $W \models  2^{\aleph_0} < \lambda
$.

To see this, note that by $(b)$ and [3, Theorem VII.4.2],
$K_{\lambda}(W)=K_{\lambda}(V)$. Then by [3, conclusion VII.4.3(3)],
$(W,V)$ satisfies the strong $(\aleph_1, \lambda)-$covering
property. On the other hand $\lambda >
([(2^{\aleph_0})^{W}]^{+})^{V}$, and hence by Theorem 3.2,
$\lambda \leq (2^{\aleph_0})^{V}= | (2^{\aleph_0})^{W} |
^{V} < \lambda.$ Contradiction.

\section{Some consistency results}

In this section we consider the work of Shelah and Woodin in [4]
and prove some related results.

\begin{theorem} There is a generic extension $W$ of $L$ and two
reals $a$ and $b$ such that:

$\hspace{1cm}(a)$ Both of $W[a]$ and $W[b]$ satisfy $CH$.

$\hspace{1cm}(b)$ $CH$ fails in $W[a,b]$.

Furthermore $2^{\aleph_0}$ can be arbitrary large in $W[a,b]$.
\end{theorem}
\begin{proof} Let $\lambda \geq \aleph_{5}^{\mathbf{L}}$ be regular in $\mathbf{L}$.
By [4, Theorem 1] there is a pair $(W,V)$ of generic extensions
of $\mathbf{L}$ such that:

$\hspace{1cm}\bullet$ $(W,V)$ satisfies $(1)_{\aleph_0}$.

$\hspace{1cm}\bullet$ $V \models  2^{\aleph_0}= \lambda.
$

Let $V=W[r]$ where $r$ is a real. Working in $V$, let
$\textrm{P}=Col(\aleph_0, \aleph_1)$ and let $G$ be
$\textrm{P}-$generic over $V$. In $V[G]$ the set $\{D \in V: D$ is
open dense in $Add(\omega, 1) \}$ is countable, hence we can
easily find two reals $a$ and $b$ in $V[G]$ such that both of $a$
and $b$ are $Add(\omega, 1)-$generic over $V$, and $r \in L[a,b].$
Then the model $W$ and the reals $a$ and $b$ are as
required.
\end{proof}

Note that for $\kappa > \aleph_0,$ by Theorem 3.3(4) we can not
expect to obtain [4, Theorems 1 and 2] without assuming the
existence of $0^{\sharp}.$ However it is natural to ask whether it
is possible to extend them under the assumption $``0^{\sharp}$
exists''. The following result (Cf. [1, Lemma 1.6]) shows that for
$\kappa > \aleph_0,$ there is no pair $(W,V)$ satisfying
$(1)_{\kappa}$ such that $0^{\sharp} \notin W \subseteq
\mathbf{L}[0^{\sharp}]$ and $W$ and $\mathbf{L}[0^{\sharp}]$ have the same cardinals
$\leq \kappa^{+}.$

\begin{theorem} Let $\kappa = \aleph_{1}^{V}.$ If $0^{\sharp}$
exists and $M$ is an inner model in which $\kappa_{\mathbf{L}}^{+}$ is
collapsed, then $0^{\sharp} \in M.$
\end{theorem}
\begin{proof} Let $I$ be the class of Silver indiscernibles. There
are constructible clubs $C_{n}, n < \omega,$ such that $I \cap
\kappa= \bigcap_{n< \omega}C_{n}.$ if $\kappa_{\mathbf{L}}^{+}$ is
collapsed in $M$, then in $M$ there is a club $C$ of $\kappa$,
almost contained in every constructible club. Hence $C$ is almost
contained in the $\bigcap_{n< \omega}C_{n}$ and hence in $I \cap
\kappa.$ It follows that $0^{\sharp}\in M.$
\end{proof}
We now prove a strengthening of Theorem 4.1 under stronger
hypotheses.

\begin{theorem} Suppose $cf(\lambda) > \aleph_0$, there are
$\lambda-$many measurable cardinals and $GCH$ holds. Then there is
a cardinal preserving generic extension $W$ of the universe and
two reals $a$ and $b$ such that:

$\hspace{1cm}(a)$ The models $W, W[a], W[b]$ and $W[a,b]$ have the
same cardinals.

$\hspace{1cm}(b)$ $W[a]$ and $W[b]$ satisfy $GCH$.

$\hspace{1cm}(c)$ $W[a,b] \models  2^{\aleph_0}= \lambda
.$
\end{theorem}
\begin{proof} By [4, Theorem 4] there is a pair $(W,V)$ of
cardinal preserving generic extensions of the universe such that:

$\hspace{1cm}\bullet$ $(W,V)$ satisfies $(1)_{\aleph_0},
(2)_{\aleph_0}$ and $(3)_{\aleph_0}$.

$\hspace{1cm}\bullet$ $V \models  2^{\aleph_0}= \lambda
.$

Working in $V$, let $\textrm{P}=Col(\aleph_0, \aleph_1)$ and let
$G$ be $\textrm{P}-$generic over $V$. As in the proof of Theorem
4.1 we can find two reals $a^*$ and $b^*$ such that $a^*$ is
$Add(\omega, 1)-$generic over $V$ and $b^*$ is $Add(\omega,
1)-$generic over $V[a^*],$ where $Add(\omega, 1)$ is the Cohen
forcing for adding a new real. Note that $V[a^*]$ and $V[a^*,
b^*]$ are cardinal preserving generic extensions of $V$. Working
in $V[a^*, b^*]$ let $\langle k_N: N<\omega \rangle$ be an
increasing enumeration of $\{N: a^{*}(N)=0\}$ and let $a=a^*$ and
$b=\{N: b^{*}(N)=a^{*}(N)=1 \} \cup \{k_N:r(N)=1 \}$ where $V=W[r]
.$ Then clearly $r \in \mathbf{L}[\langle k_N: N<\omega \rangle,
b]\subseteq \mathbf{L}[a,b]$ as $r=\{N: k_N \in b \}.$

We show that $b$ is $Add(\omega, 1)-$generic over $V$. It suffices
to prove the following:

$\hspace{1.4cm}$ For any $(p,q) \in Add(\omega,
1)*\lusim{A}dd(\omega, 1)$ and any open dense subset $D$

$(*)$$\hspace{1cm}$ $ \in V$ of $Add(\omega, 1)$ there is
$(\bar{p},\bar{q}) \leq (p,q)$ such that $(\bar{p},\bar{q}) \vdash
  `` \dot{b}$ extends

$\hspace{1.5cm}$ some element of $D $''.

Let $(p,q)$ and $D$ be as above. By extending one of $p$ or $q$ if
necessary, we can assume that $lh(p)=lh(q)$. Let $\langle k_N: N<M
\rangle$ be an increasing enumeration of $\{N<lh(p): p(N)=0\}.$
Let $s: lh(p) \rightarrow 2$ be such that considered as a subset
of $\omega,$

\begin{center}
$s=\{ N<lh(p): p(N)=q(N)=1 \} \cup \{k_N: N<M, r(N)=1 \}.$
\end{center}

Let $t \in D$ be such that $t \leq s.$

\begin{claim} There is an extension $(\bar{p},\bar{q})$ of $(p,q)$
such that $lh(\bar{p})=lh(\bar{q})=lh(t)$ and
\begin{center}
$t=\{ N<lh(t): \bar{p}(N)=\bar{q}(N)=1 \} \cup \{k_N: N<\bar{M},
r(N)=1 \},$
\end{center}
 where $\langle k_N: N<\bar{M}
\rangle$ is an increasing enumeration of $\{N<lh(\bar{p}):
\bar{p}(N)=0\}$.
\end{claim}
\begin{proof} Extend $p,q$ to $\bar{p}, \bar{q}$ of length $lh(t)$
so that for $i$ in the interval $[lh(s),lh(t))$
\begin{itemize}
\item $\bar{p}(i)=1$, \item $\bar{q}(i)=1$ iff $i \in t.$
\end{itemize}

Then
\begin{center}
$t=\{ N<lh(t): \bar{p}(N)=\bar{q}(N)=1 \} \cup \{k_N: N<M, r(N)=1
\}.$
\end{center}

But (using our definitions) $M =\bar{M}$ so

\begin{center}
$t=\{ N<lh(t): \bar{p}(N)=\bar{q}(N)=1 \} \cup \{k_N: N<\bar{M},
r(N)=1 \}.$
\end{center}

as desired.
\end{proof}
It follows that

\begin{center}
$(\bar{p},\bar{q}) \vdash  \dot{b}$ extends t
\end{center}
and $(*)$ follows.

It follows that $a$ and $b$ are $Add(\omega,1)-$generic over $W$
and $r \in \mathbf{L}[a,b].$ Hence the model $W$ and the reals $a$ and $b$
are as required and the theorem follows.
\end{proof}
\begin{remark} The above kind of argument is widely used in [2] to
prove the genericity of a $\l-$sequence of reals over $Add(\omega,
\l)$, the Cohen forcing for adding $\l-$many new reals.
\end{remark}
\section{Open problems}

We close this paper by some remarks and open problems.
Our first problem is related to Vanliere's Theorem.

\begin{problem} Find the least $\kappa$ such that there are $X
\subseteq \kappa$ and $a \subseteq \omega$ such that $\mathbf{L}[X] \models
  ZFC+GCH  , \mathbf{L}[X]$ and $\mathbf{L}[X,a]$ have the same
cardinals and $\mathbf{L}[X,a] \models  2^{\aleph_0} >
\aleph_1$
\end{problem}
Now consider the following property:

(*):$\hspace{1cm}$ If $\textrm{P}$ is a non-trivial forcing notion
and $G$ is $\textrm{P}-$generic over $V$,

$\hspace{1.5cm}$ then for any cardinal $\chi \geq \aleph_{2}^{V}$
and $x \in H^{V[G]}(\chi),$ there is $N \prec$

$\hspace{1.5cm}$ $ \langle H^{V[G]}(\chi), \in, <_{\chi}^{*},
H^{V}(\chi) \rangle$, such that $x \in N$ and $N \cap H^{V}(\chi)
\in$

$\hspace{1.5cm}$ $V,$ where $<_{\chi}^{*}$ is a well-ordering of
$H^{V}(\chi).$

Note that if $(*)$ holds, and $p \in G$ is such that $p$ forces
$``x \in N$ and $N \cap H^{V}(\chi) \in V$'' and decides a value for
$N \cap H^{V}(\chi)$, then $p$ is $(N \cap H^{V}(\chi),
\textrm{P})-$generic . Using Shelah's work on strong covering
property we can easily show that:

$\hspace{1cm}(a)$ If $0^{\sharp} \notin V,$ then $(*)$ holds.

$\hspace{1cm}(b)$ If in $V$ there is no inner model with a
measurable cardinal, then $(*)$ holds for

$\hspace{1.5cm}$any cardinal preserving forcing notion
$\textrm{P}.$

Now we state the following problem:

\begin{problem} Suppose $0^{\sharp} \in V.$ Does $(*)$ fail for
any non-trivial constructible forcing notion $\textrm{P}$ with
$o^{\mathbf{L}}(\textrm{P}) \geq \omega_{1}^{V},$ where $o^{\mathbf{L}}(\textrm{P})$
is the least $\beta$ such that forcing with $\textrm{P}$ over $\mathbf{L}$
adds a new subset to $\beta.$
\end{problem}
This problem is motivated by the fact that if $0^{\sharp} \in V,$
then for any non-trivial constructible forcing notion
$\textrm{P},$ forcing with $\textrm{P}$ over $V$ collapses
$o^{\mathbf{L}}(\textrm{P})$ into $\omega$ (Cf. [5]).

\begin{problem} Assume $V \models  GCH ,
\lambda$ is a cardinal in $V, A \subseteq \lambda$ and $V[A]$ is a
model of set theory with the same cardinals as $V$. Can we have
more than $\lambda-$many reals in $H^{V}(\lambda^{+})[A].$
\end{problem}
Let us note that if $\lambda$ is regular in $V$, then the answer
is no, and if $\lambda$ has countable cofinality in $V$, then the
answer is yes. Also if there is a stationary subset of
$[\lambda]^{\leq \aleph_0}$ in $V[A]$ of size $\leq \lambda,$ then
the answer is no (Cf. [3, Theorem VII.0.5(1)]). Let us note that the
Theorem as stated in [3] is wrong. The conclusion should be: There
are $\leq \lambda$ many reals in $H^{V}(\lambda^{+})[A])$.

Concerning Problem 5.3, the main question is for $\lambda=
\aleph_{\omega_1}.$ We restate it for this special case.

\begin{problem} (Mathias). Is Problem 5.3 true for $\lambda=
\aleph_{\omega_1}.$
\end{problem}

\end{document}